\documentclass[paper=letter, fontsize=10pt,leqno]{scrartcl}
\usepackage[body={5.35in,7.5in}]{geometry} 
\usepackage[english]{babel}				
\usepackage{amsmath,tabu}				
\usepackage{amsfonts}
\usepackage{amsthm}
\usepackage{latexsym}
\usepackage{hyperref}
\usepackage[section]{placeins}
 \usepackage{verbatim}
\usepackage[pdftex]{graphicx}				
\usepackage{url}
\usepackage{amssymb}
\usepackage{bbm}
\usepackage{MnSymbol}
\usepackage{pbox}
\usepackage{enumerate}  
\usepackage[pdftex]{color}
\usepackage{eucal}
\usepackage{amscd}
\usepackage[format=plain]{caption}
\usepackage{wrapfig}
\usepackage{float}
\usepackage{bigints}

\usepackage{sectsty}					
\allsectionsfont{\centering \normalfont\scshape}	

\usepackage{mathtools}
\DeclareCaptionStyle{ruled}{labelfont=bf,labelsep=space}
\usepackage[]{forest}
\usepackage{amsthm}
\usepackage{multirow}

\makeatletter
\newtheorem*{rep@theorem}{\rep@title}
\newcommand{\newreptheorem}[2]{%
\newenvironment{rep#1}[1]{%
 \def\rep@title{#2 \ref{##1}}%
 \begin{rep@theorem}}%
 {\end{rep@theorem}}}
\makeatother

\newreptheorem{theorem}{Theorem}

\newtheorem{theorem}{Theorem}

\newtheorem{corollary}[theorem]{Corollary}
\newtheorem{conjecture}[theorem]{Conjecture}

\newtheorem*{proposition*}{Proposition}

\newtheorem*{remark}{Remark}
\newtheorem*{thm:bounce}{Theorem \ref{thm:bounce}}
\DeclareMathOperator{\sgn}{sgn}

\providecommand{\keywords}[1]
{
  \small	
  \textbf{\textit{Keywords---}} #1
}

\title{
		\vspace{-1in} 	
		\usefont{OT1}{bch}{b}{n}
		\normalfont \normalsize \textsc{} \\ [25pt]
		\huge   Dynamics of the no-slip Galton board
}

\date{}

\author{\normalfont \large
J. Ahmed\footnote{\scriptsize Department of Mathematics, University of Delaware, Ewing Hall, Newark, DE 19711}, 
\ T. Chumley\footnote{\scriptsize Department of Mathematics and Statistics, Mount Holyoke College, 50 College St, South Hadley, MA 01075},
\ S. Cook\footnote{\scriptsize Department of Mathematics, Tarleton State University, Box T-0470, Stephenville, TX 76401},
\ C. Cox\footnotemark[2],
\\
\large
\ H. Grant\footnotemark[3], 
\ N. Petela\footnotemark[3],
\ B. Rothrock\footnotemark[3],
\ R. Xhafaj\footnotemark[3]
}

\date{\today}

\begin{document}

\maketitle

\begin{abstract}
\begin{center}
 Abstract \end{center}
\bigskip
{\small  
The ideal Galton board and Lorentz gas billiard models have been studied numerically and analytically primarily in settings where friction and rotational velocity are neglected. We eliminate these simplifying assumptions and  study the resulting dynamics of a more general model using no-slip collisions, in which particles rotate and may exchange linear and angular momentum at collisions while adhering to certain conservation laws. 
Using numerical experiments and phase portrait analysis we show that (in contrast to specular dispersing billiards) regularity persists when a small force is introduced while (consistent with specular billiards) under a stronger force new structure including invariant regions may arise. We also show analytically that with the introduction of an external force periodicity proliferates, with new types of periodic orbits not present in the no-force case.
}
\end{abstract}


\section{Introduction}
\label{intro
}

In Chapter V of his 1889 treatise {\em Natural Inheritance}  \cite{galton},  Francis Galton wrote:
\begin{quote}
    I know of scarcely anything so apt to impress the imagination as the wonderful form of cosmic order expressed by the ``Law of Frequency of Error.'' The law would have been personified by the Greeks and deified, if they has known of it. It reigns with serenity and in complete self-effacement amidst the wildest confusion. The huger the mob, and the greater the apparent anarchy, the more perfect is its sway. It is the supreme law of Unreason.
\end{quote}

In the late nineteenth century, Galton found the physical manifestation of such Gaussian error laws, later codified and demystified as central limit theorems, 
in his eponymous board (also known as a ``bean machine'' or ``quincunx'') in which balls dropping from a central location above 
are scattered by regularly spaced pegs or nails, collectively coming to rest according to a normal distribution at the bottom. 
At the beginning of the twentieth century, consequent to physicists' search for an illuminative underlying model for  thermal
and electrical conductivity, Hendrik
Lorentz \cite{Lorentz} introduced a model in which the ions in a metal
were considered as an extended array of fixed disks, between which an electron would move freely until colliding with a scattering disk, with no loss of energy. Both the Galton board (a system with external force and unbounded total energy) and the Lorentz gas (a system with no external force and fixed total energy) are examples of billiard dynamical systems, which have proven broadly applicable and continue to be widely studied this century.
The Galton board has been approached using numerical simulations \cite{kozlov, Moran} and  heuristic estimates \cite{KR}, as well as analytically, as in 
\cite{CD} which  considers an idealized Galton board bounded above and continuing infinitely downward with regularly spaced disk scatterers in which friction and the rotational velocity of the balls are ignored. All of the above classical and modern examples, with particles moving freely until colliding at boundaries  with mirror-like reflections, are examples of \emph{specular billiards}, the most widely used billiard collision model.  

\begin{figure}[htbp]
    \centering
    \includegraphics[width=\textwidth]{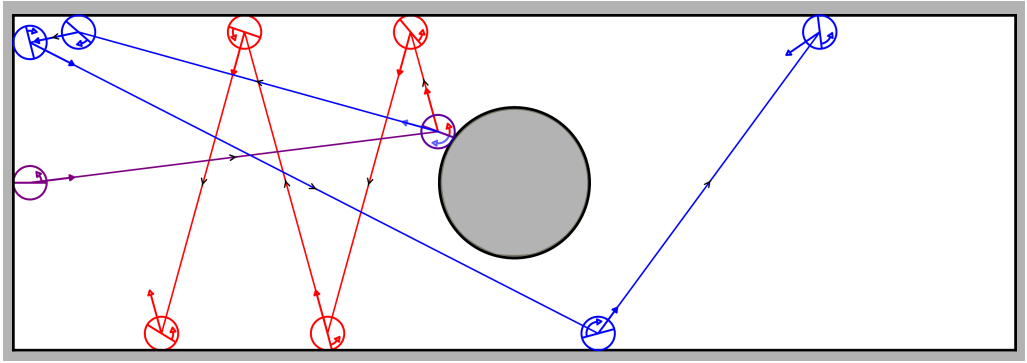}
    \caption{\small 
    Initially identical specular  and no-slip  orbits, starting middle left, in a rectangular billiard with a single scatterer. In the no-slip collisions the rotational velocity changes and the particle exits at an angle different from its entrance, yielding marked differences in the dynamics. In the specular case, the particle reflects from the scatterer nearly vertically, while in the no-slip case it heads towards the upper left corner. Internal particle arrows indicate rotational velocity, unchanged in the specular case but changing direction and magnitude after no-slip collisions.}
    \label{fig:specnoslip}
\end{figure}

Our focus, however,  is on the less well understood alternative \emph{no-slip} or \emph{nonholonomic} billiard model, 
primarily in dimension two, which allows non-dissipative friction to effect an exchange between rotational and linear velocities at collisions. Figure \ref{fig:specnoslip} shows an example of the divergence between the standard billiard model where the angles of incidence and reflection coincide (which we henceforward refer to as the \emph{specular} model) and the no-slip model for one trajectory.  This model of incorporating particle rotation was studied at least since Richard Garwin \cite{Garwin} used it to explain the motion of an ultraelestic Super Ball in 1969. Garwin's model was ostensibly three dimensional, but assumed the axis of rotation of the ball was orthogonal to the plane of motion, effectively reducing to two spacial dimensions.  
In \cite{gutkin} a planar model  is derived and the foundational dynamical fact is shown that in an infinite strip or \emph{channel} all colliding trajectories are bounded. 
A general geometric description of no-slip collisions in any spatial dimension is given in \cite{CF}, including specifically the case with unrestricted rotation in dimension three. The model is modified in 
\cite{cross} and 
\cite{hefner}  to consider dissipative cases, where energy is lost at collisions according to a tangential or normal coefficient of restitution. We will restrict our inquiry to the ideal model. 

Independently of the above derivations, which arise purely from considering conservation laws, 
in \cite{BKM} it was shown that the model arises in the limit of nonholonomic rolling on a cylinder or ellipsoid  flattened into an infinite channel or ellipse, respectively. This connection can be extended to the nonholonomic rolling limits of a broad class of manifolds \cite{CFBZ}. In light of this correspondence, it is not surprising that in certain cases small no-slip bounces on a cylinder approximate nonholonomic rolling \cite{CCCF}.

\begin{figure}[htbp]
    \centering
    \includegraphics[width=.97\textwidth]{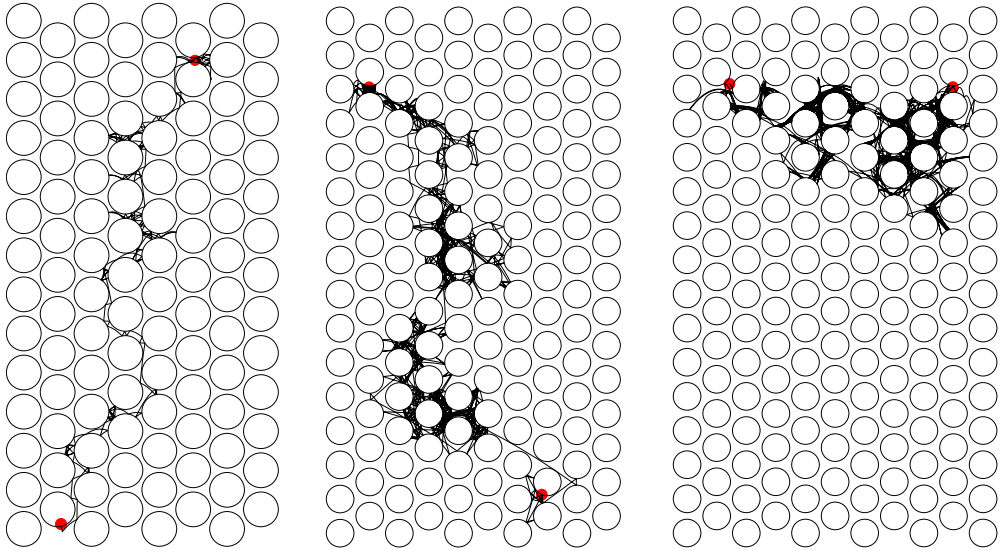}
    \caption{\small With a downward force, in the specular Galton board (left, $\gamma=0$) the particle reaches a fixed vertical position with probability 1, and the empirical distribution of its horizontal position will approach the Gaussian distribution as the number of particles goes to infinity. For a no-slip Galton board some orbits may be trapped (possibly the one on the right, $\gamma=3$) or merely slowed down (center, $\gamma=1/\sqrt{2}$).
    Dots indicate initial and final positions of the particle. }
    \label{fig:galtonboard}
\end{figure}

The rich nature of the dynamics of no-slip collisions hinted at above leads us to the present investigation,
where we study the Galton board and Lorentz gas models using rotating particles which collide with scatterers according to the no-slip collision law.
In this alternative law, collisions are influenced by rotational inertia and consequently by the mass distribution of the particles. In fact, recent work  shows that the dynamics of no-slip billiards may be altered in an essential way, for a fixed billiard table and particles of fixed radii, by merely varying the mass distribution of a radially symmetric colliding particle \cite{ACW}. More precisely, then, we consider a family of Galton boards parametrized by a mass distribution constant of the colliding particle. We assume a disk-shaped particle with radially symmetric mass distribution. 
Such systems may be identified by a parameter $\gamma=\sqrt{2\lambda}/R$ where $R$ is the particle radius and $\lambda$ is the second moment of inertia. 
Then  $\gamma=0$ corresponds to point masses, dynamically equivalent\footnote{More specifically, in the $\gamma=0$ no-slip case the rotation is exactly reversed while in the specular case the rotation is unaltered. From the standpoint of \cite{hefner}, the latter is the case with tangential coefficient of restitution 1 and moment 0, while the latter is the case of tangential coefficient of restitution -1 and any moment. 
However, the essential point is the two yield identical spatial trajectories given by specular collisions.} to the specular model, $\gamma = 1/\sqrt{2} \approx 0.707$ corresponds to uniform mass, $\gamma=1$ corresponds to ring-like particles with all the mass on the rim, and  as the moment gets large $\gamma$ approaches infinity.
In Theorems \ref{thm:bounce} and \ref{thm:periodic} we will see that periodic behavior will be tied to certain ratios of velocities aligning with the parameter $\gamma$.
Figure \ref{fig:galtonboard} shows three trajectories in a Galton board, the first reflecting with specular collisions and the other two modeled by no-slip collisions with $\gamma>0$.

\begin{figure}[htbp]
    \centering
    \includegraphics[width=.77\textwidth]{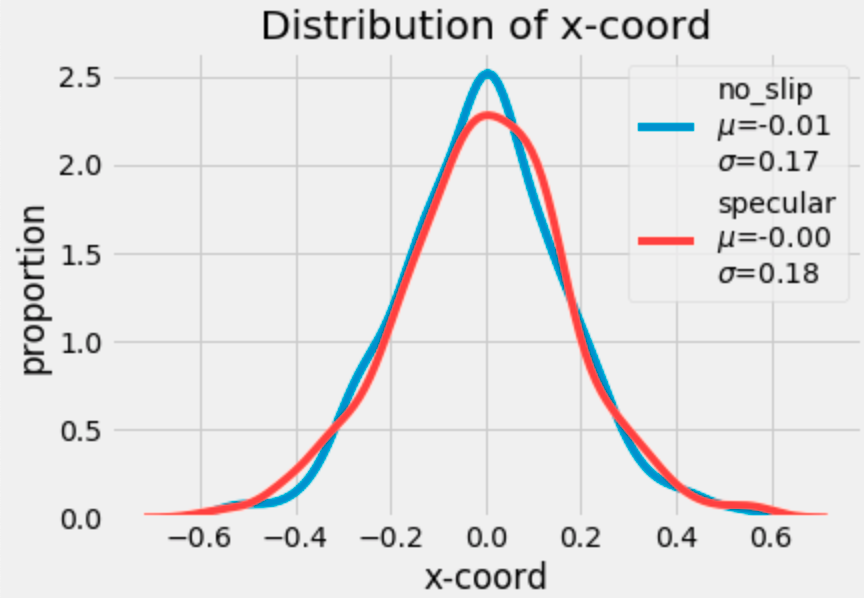}
    \caption{\small The no-slip Galton board appears to yield the same normal distribution of terminal horizontal displacements that the standard Galton board displays, however, the large number of trapped particles suggests more complicated behavior cannot be ruled out.}
    \label{fig:normal}
\end{figure}

A natural first experiment, hearkening back to Galton, is to consider the distribution of particles reaching a fixed height in a  no-slip Galton board and compare it to the classical case. Figure \ref{fig:normal} gives the distribution of the terminal horizontal position of particles
for both the  specular Galton board and the modified no-slip board with uniform mass particles. In both cases, 100,000 particles were dropped with unit speed and a random initial direction, and final horizontal displacements were recorded for all particles which reached the terminal height by a fixed time. On the surface, it appears that the behavior is unchanged for no-slip collisions, as the result gives an apparently Gaussian distribution with a similar standard deviation. However, while almost all of the particles in the specular Galton board reach the terminal height before the termination time, a significant percentage (in the given experiment, roughly forty percent) of the particles in the no-slip case have not arrived at the end of the experiment, either because their descent is slowed or because they are stuck in invariant regions of the phase space. Even when the termination time is extended to the reasonable limits of calculation this phenomenon persists.

For this reason, we postpone the apparently subtle questions of long-term statistical behavior and in this paper focus on the dynamics that arise in no-slip billiards under an external force. Towards that end we focus on basic questions which remain of interest even for specular billiards with no force, and which are largely unanswered for no-slip billiards with force. First, when are the orbits periodic? Second, if not periodic,
when are the orbits bounded? Linking the two questions is the third question of whether
a given periodic orbit is Lyapunov stable, producing a bounded corridor containing an
invariant region. Generally speaking, when no force is present periodic points and surrounding invariant regions are much more common for no-slip billiards than in standard specular billiards. See Figure \ref{fig:pentagons}.

\begin{figure}[htbp]
    \centering
    \includegraphics[width=\textwidth]{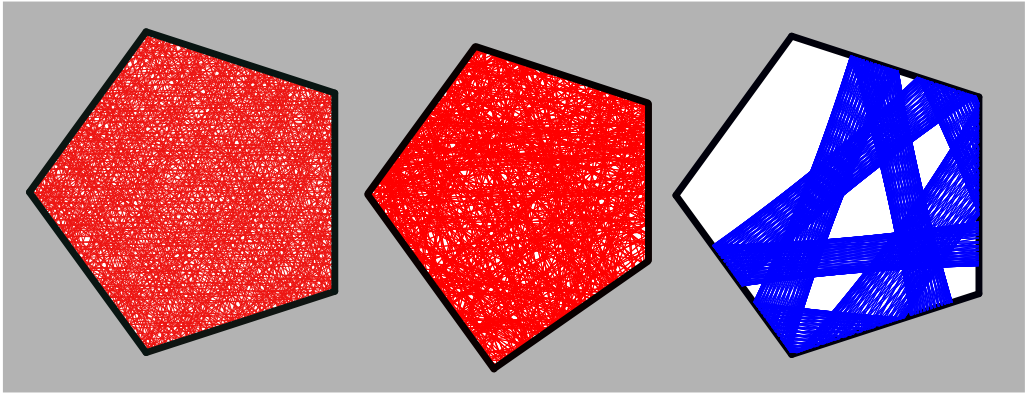}
   
    \caption{\small  A generic orbit of a specular regular pentagon billiard (left), a specular non-regular pentagon (center), and a no-slip pentagon (right) where the orbit is restricted to a corridor near a $14$-periodic point. Here there is no external force present. }
    
    \label{fig:pentagons}
\end{figure}

As the boundedness of the infinite strip was the first dynamical no-slip result understood in the no force case, we begin by considering a channel with either a parallel or orthogonal force added. The parallel case, with the force along the channel, intuitively corresponds to ball dropped down an infinite well, a case that has already been studied. For the restricted three dimensional case with orthogonal rotational axis,  
\cite{hefner} gives a closed form expression for the height of the dropped ball, noting that it is bounded (unless, of course, dropped straight down avoiding the walls altogether) and that the range as it bounces up and down depends on the moment of inertia. This boundedness was later shown to hold for a channel between two hyperplanes in any dimension, and in fact for any cylinder  \cite{CCCF}.

Turning then to the orthogonal force case, if the energy is sufficient to reach the upper wall for a given initial velocity, it is straightforward to show that the motion remains bounded by considering tangential and normal components, applying the argument for the channel with no force to the tangential motion. Suppose instead that the energy is not sufficient to reach the upper wall, or equivalently that there is a single boundary line and the table is a half plane with orthogonal force. Intuitively, this is simply a ball bouncing on the ground. We show that if the ratio of angular to horizontal velocity precisely matches the mass distribution constant $\gamma$, the system will be  2-periodic and therefore bounded. In all other cases it will exhibit a drift rendering it unbounded. See Figure \ref{fig:bouncingball}.

\begin{figure}[htbp]
    \centering
    \includegraphics[width=1\textwidth]{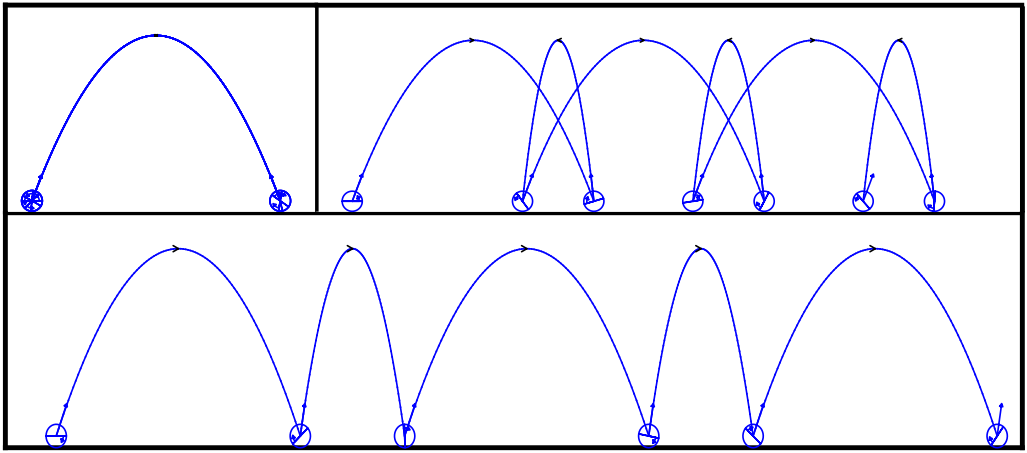}
    \caption{\small When the ratio of tangential velocity to rotational velocity equals the mass distribution constant $\gamma$, a no-slip ball bounces back and forth in a 2-periodic motion (top left). Otherwise it will drift, either in a back and forth motion with net drift (top right) or steadily (bottom).
        }
    \label{fig:bouncingball}
\end{figure}

More precisely, 
consider a no-slip billiard table consisting of the upper half plane, with particles of mass distribution $\gamma$. 
Let $\mathbf{q}_0=(x_0, y_0, z_0)$ 
and $\mathbf{p}_0=\dot{\mathbf{q}}_0$
be the initial position and velocity of an arbitrary 
trajectory, with $x_0$ the rotational position or orientation (normalized so that the energy associated with the velocity $\mathbf{p}$ is proportional to $\| \mathbf{p} \|$ under the standard Euclidean norm by letting $x_0=\gamma R \theta$, where $R$ is the particle radius and $\theta$ the rotation in radians), $y_0$ the direction parallel to the boundary, and $z_0$ the inward normal, hence upward. More generally, we let $\mathbf{q}_n$ and $\mathbf{p}_n$ be the position and velocity immediately after the $n$th collision. Without losing any generality, we require $y_0=0$ and $\dot{z}_0>0$, and we suppose there is a constant force of $\mathbf{g}=(0,0,-g)$, acting on particles during free flight between collisions. 
Define \emph{tangential-rotational ratio} as $\Delta=\dot{y}_0/\dot{x}_0$.
While in a general table this ratio would vary with each collision, in this case of the no-slip half-plane, the tangential-rotational ratio $\dot{y}_n/\dot{x}_n$ at the $n$th collision is constant in $n$.
We may classify the possible trajectories of the no-slip bouncing ball as follows. 

\begin{theorem}[No-slip bouncing in the half-plane]
For any trajectory in
the no-slip half plane,
the velocity is $2$-periodic. That is, $\mathbf{v}_i=\mathbf{v}_{i+2}$ for any integer $i$. If $\Delta=\gamma$,
then the phase space orbit is $2$-periodic, so that $(\mathbf{q}_i,\mathbf{v}_i)= (\mathbf{q}_{i+2},\mathbf{v}_{i+2})$ for any integer $i$. If $\Delta \neq \gamma$, the lateral displacement results in a constant net drift after every two collisions, and the position is unbounded.

\label{thm:bounce}
\end{theorem}
\begin{remark}
In Section \ref{sec:bounce} we show as a corollary that this result generalizes naturally to a billiard with a single hyperplane boundary and an orthogonal force in any dimension.
\end{remark}

For specular billiards, periodic orbits are dense in rational polygons \cite{BGKT} and can be created, for example, by constructing tables with certain arrangements of arcs of circles \cite{CZ15}. In the case of no-slip billiards, periodic orbits appear in more generic settings. For example, regardless of rationality considerations, the set of periodic and quasiperiodic points appear to comprise a full measure set in the phase space of any no-slip polygon. Evidence of this appears in the highly structured phase portraits shown in Figure \ref{fig:vppregpoly}, though a detailed discussion is left for future study.

\begin{figure}[htbp]
    \centering
    \includegraphics[width=1\textwidth]{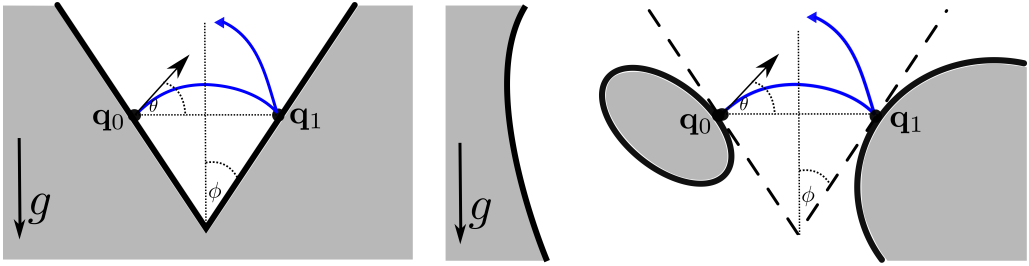}
    \caption{\small (Left) For a wedge billiard with total angle $2\phi$ and a downward force $\mathbf g$, we consider orbits with launch angle $\theta$. (Right) More generally, periodicity results extend to any billiard where $\mathbf{q}_0$ and $\mathbf{q}_1$ have a corresponding tangent wedge  accessible with no internal obstructions.
        }
    \label{fig:newwedge}
\end{figure}

Our focus now turns to periodicity in no-slip tables with force beyond the previously discussed half-plane.
We begin with a discussion of periodicity in systems we call \emph{no-slip wedges}.
Such tables are formed by the interior of two planar lines that meet an angle less than $\pi$.
We refer to angle between one of these lines and the bisector of the wedge as the \emph{wedge angle}.
We generally assume that the table is embedded in the plane so that the bisector is oriented vertically.
See Figure \ref{fig:newwedge}, left.
For any two boundary points $\mathbf{q}_0$ and $\mathbf{q}_1$ directly opposite on a no-slip wedge with angle $\phi$,
it is known, for systems without external force, that choosing the appropriate rotational velocity ensures that the orbit along a trajectory going from one boundary point to the other will be $2$-periodic \cite{CFZ}.
More precisely, if we choose coordinates such that $x$ is the rotational position of the particle, $z$ the direction of the wedge bisector, and $y$ the direction orthogonal to $z$, then 
the orbit  with initial velocity $\mathbf{q}_0=(\dot{x}_0,\dot{y}_0,\dot{z}_0)$ will have period two if and only if $\dot{z}_0=0$ and
\begin{equation}
\gamma=-\frac{\dot{y}_0}{\dot{x}_0} \sin\phi.
\label{eq:spw}
\end{equation}

The present study shows that 2-periodic orbits continue to exist  under an analogous condition when we introduce a downward force $\mathbf{g}=(0,0,-g)$, given in the newly defined coordinates adapted to the wedge.
(Recall that the first coordinate of $\mathbf g$ is the rotation, which is necessarily zero but which we include for consistency.)
In this case no horizontal periodic trajectory is possible. 
However, we can generalize Equation \eqref{eq:spw} to any upward and inward trajectory. 
That is, periodic trajectories exist for any initial state $(\mathbf{q}_0, \mathbf{p}_0)$ with $\dot{z}_0>0$ and $\sgn{y}_0=-\sgn\dot{y}_0$, for any wedge with angle $0<\phi<\pi/2$ and force parallel to bisector under a condition, to be given precisely below,
in terms the angle $\theta$, which we call the \emph{launch angle},
between the vector $\mathbf p_0$ and the vector connecting $\mathbf q_0$ to $\mathbf q_1$.
In fact, since no-slip collisions are completely determined by the particle's velocity at the point of contact, Equation \eqref{eq:spw} can be generalized to tables beyond wedges.
We call a pair of boundary points $\mathbf q_0$ and $\mathbf q_1$ \emph{accessible} if the lines tangent to the boundary at $\mathbf q_0$ and $\mathbf q_1$ form an upward (or even downward) wedge with wedge angle $0<\phi<\pi/2$, such that the wedge has no internal obstructions.
That is, the free path between $\mathbf q_0$ and $\mathbf q_1$  within the wedge must not intersect the boundary of the table at any other point. We refer to the corresponding wedge between accessible boundary points as the \emph{tangential wedge}.
See Figure  \ref{fig:newwedge}, right.
Finally, we define a 2-periodic orbit between $\mathbf q_0$ and $\mathbf q_1$ to be \emph{path-reversing} if the trajectory traced along the orbit from $\mathbf q_0$ to $\mathbf q_1$ is the same as that along the orbit from $\mathbf q_1$ and $\mathbf q_0$.
We are now ready to characterize 2-periodic orbits in the following theorem.
The proof is given in Section \ref{sec:bounce}. 

\begin{theorem}(Periodicity in no-slip billiards with force) 
Suppose that $\mathbf{q}_0$ and $\mathbf{q}_1$  are accessible boundary points on a no-slip billiard  directly opposite on their tangential wedge, with force along the bisector $\mathbf{g}=(0,0,-g)$.
Then for any direction inward relative to the 
tangential wedge with
launch angle $\theta>0$ 
(hence $\dot{z}_0=\sin\theta>0$) there is a path-reversing $2$-periodic orbit between
$\mathbf{q}_0$ and $\mathbf{q}_1$ if and only if
\begin{equation}
\label{eq:2per}
    \gamma=\frac{v}{\dot{x}_0}(\sin\theta\cos\phi-\cos\theta\sin\phi)
\end{equation}
where 
\begin{equation}
\label{eq:vel}
v=2\sqrt{\frac{gd}{\sin{2\theta}}} 
\end{equation}
with $d=\| \mathbf{q}_0-\mathbf{q}_1 \|$, the distance between boundary points.
\label{thm:periodic}
\end{theorem}

\begin{figure}[htbp]
    \centering
    \includegraphics[width=\textwidth]{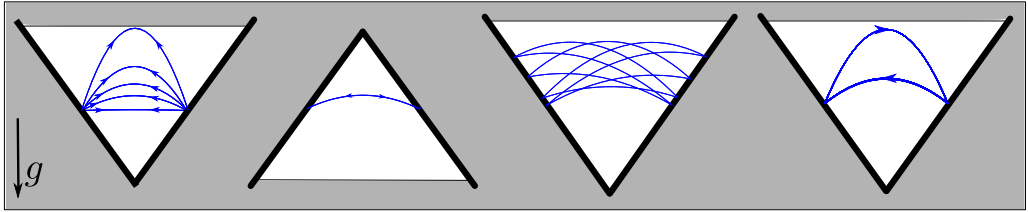}

    \caption{\small 
    Wedge billiards with force shown downward. Left: Four $2$-periodic path-reversing orbits, which may occur whenever the velocity has an upward component. Left center: These orbits exist even for wedges opening downward.  
    Right center: An isolated period 8 orbit. Right: A $2$-periodic orbit that is not path-reversing.    }
    \label{fig:wedge}
\end{figure}

Recalling that trajectories colliding with the wall in an infinitely deep no-slip vertical channel are bounded, it is natural to wonder if this extends even in the case of a wedge opening downward. The following corollary, which will be immediate from the proof of the theorem with appropriate sign changes, answers this question in the affirmative. Figure \ref{fig:wedge} (left, left center) gives experimental confirmation of the theorem and the corollary.

\begin{corollary}
Suppose the conditions of Theorem \ref{thm:periodic} hold except with upwardly oriented force $\mathbf{g}=(0,0,g)$. Then any orbit between accessible points with initial state $\dot z_0>0$ will  be path-reversing $2$-periodic if and only if Equation \eqref{eq:2per} holds.
\label{cor:opendown}
\end{corollary}

The corollary may seem to defy gravity, but it does not violate any conservation laws and is perhaps less counterintuitive in light of the boundedness in the no-slip channel. This in turn might seem more natural when one recalls the link between no-slip bouncing the  known nonholonomic rolling phenomenon, occasionally manifested in golf balls rolling out of holes or basketballs rolling around the rim and out \cite{gualtieri}. 

These results on path-reversing $2$-periodic orbits already show that orbits in the no-slip case multifurcate into an infinite family when force is introduced, but the new possibilities for periodicity are even more widespread. For example, simulations suggest the following possibilities may arise.

\begin{conjecture}
For a no-slip billiard with an external force,
\begin{enumerate}[i.]
    \item   $2$-periodic non-path-reversing orbits exist between any two accessible points on an arbitrary table, and 
\item isolated higher order periodic orbits exist between accessible points on a wedge.
\end{enumerate}
\end{conjecture}

Numerical evidence for the conjecture is shown in Figure \ref{fig:wedge} (right center, right).
Note that neither of these types of examples have analogs when no force is present. Regarding the first, since $2$-periodic orbits appear only if Equation \eqref{eq:spw}  holds, they must necessarily be path reversing. Regarding the second, note that for a wedge of generic angle 
the simple $2$-periodic orbits  are the \emph{only} periodic orbits. For a certain dense set of angles the wedge will be \emph{persistently periodic}, having a phase space consisting entirely of periodic orbits, all of the same period.
As noted in \cite{ACW} and implicit in the proof of periodicity in \cite{CFZ},  
these two extremes are the only possibilities for no-slip billiards with no force present. Hence, no higher order periodic orbit will exist in isolation, unless a force is introduced.

The final question we consider is whether these periodic points may be  stable, creating invariant regions in phase space complicating analysis of the no-slip Galton board. Even for the case without force, an analytic demonstration of Lyapunov stability (informally, all trajectories starting in a neighborhood remain nearby) for no-slip billiards has been given only in the case of the wedge \cite{CFZ}. For scatterers with curvature, as in the cases of principle interest, only linear stability has been established. For now, then, we limit ourselves to investigating this question numerically for no-slip billiards with force.

\begin{figure}[htbp]
    \centering
    \includegraphics[width=\textwidth]{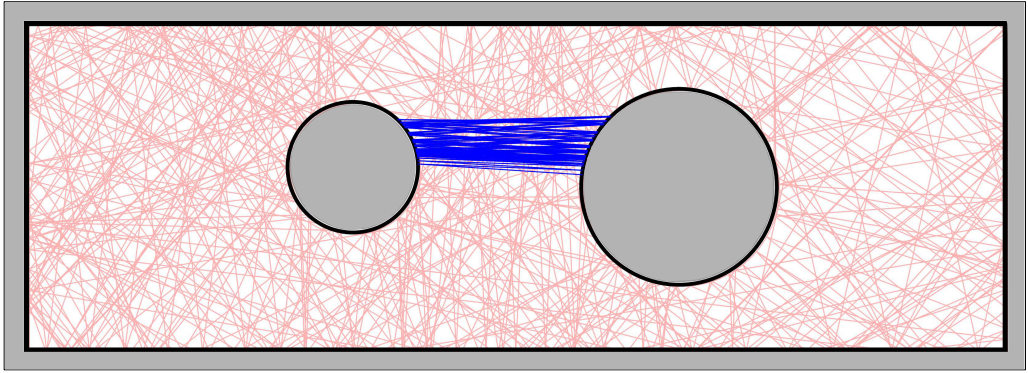}
 \caption{\small Two trajectories with the same initial conditions with no force present. The lighter trajectory in the background collides with the two disks and the outer boundaries with specular collisions, scattering. In contrast, the darker central trajectory collides with no-slip collisions and remains bounded between the disks.
        }
    \label{fig:noslipstablesinai}
\end{figure}

Specular Sinai billiards, which can be tiled across the plane to form the Lorentz gas model, are well known to be ergodic \cite{sinai}, and therefore positive measure invariant regions will not be present. No-slip Sinai billiards, however, are not ergodic and will always have (possibly very small) invariant regions in their phases spaces \cite{CFZ}, as illustrated in Figure \ref{fig:noslipstablesinai}. When even a small force is added in the standard case, the billiard map no longer preserves the standard billiard measure and therefore any discussion of ergodicity requires an alternative measure. Nonetheless, it was shown first numerically  \cite{Moran} and later demonstrated analytically \cite{Chernov2, DM}  that invariant regions may arise when a sufficient force is added. Given these two facts, it would seem likely that invariant regions might also exist in dispersing no-slip billiards with force, as we now conjecture.

\begin{conjecture}
Stable periodic orbits exist between circular scatterers for no-slip billiards with force, creating nontrivial invariant regions.  
\label{conj:inv}
\end{conjecture}

\begin{figure}[htbp]
    
    \centering
    \includegraphics[width=\textwidth]{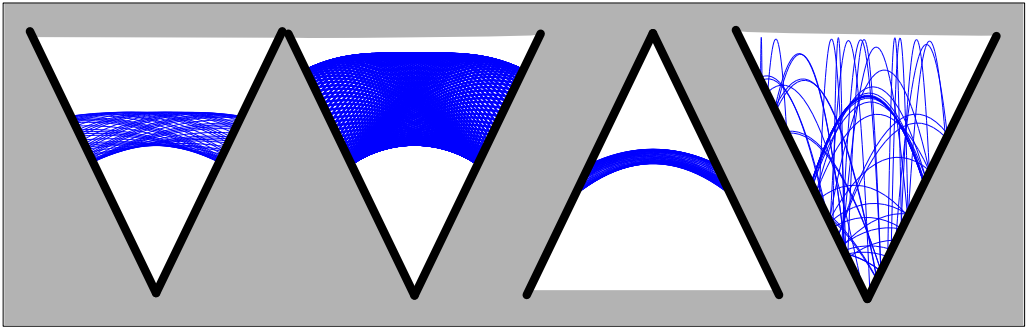}  
    \caption{\small 
For relatively small launch angles, path-reversing $2$-periodic points are surrounded by an invariant  region (left and left center), even in the case of downward opening wedges (right center). However, for large launch angles the orbits may be unstable (right).
     }
    \label{fig:2perunstablea}
\end{figure}

Looking at trajectories slightly perturbed from $2$-periodic conditions suggests that even with a force invariant regions abound, as in Figure \ref{fig:2perunstablea}.

In Section \ref{sec:inv} we pursue two further avenues of inquiry. First, we look for the limits of stability.
For specular billiards, linear stability\footnote{While linearly elliptic periodic points are generally Lyapunov stable for  specular billiards, the two are not equivalent, as pointed out in \cite{KP}. } of a $2$-periodic point is completely determined by the distance between the collisions and the curvature at collision points. This picture extends to the no-slip case \cite{W}, with the mass distribution constant $\gamma$ adding a third factor if allowed to vary. 
In the case with force, the situation becomes more complex. For example, even in the wedge where there is no curvature, fixing $\gamma$, and fixing two opposite boundary points, an orbit can pass from stable to unstable as the launch angle increases. (See Figure \ref{fig:2perunstablea}, right.) 
The second area of inquiry in Section \ref{sec:inv} is to look at a broader picture by creating phase portraits, which will allow us to gain information about the stability of a large sampling of orbits instead of a single orbit. First, however, we give a brief introduction to no-slip billiards and their phase portraits.

\section{No-slip billiards and their phase spaces}
\label{sec:back}

In this section we give the details of the no-slip collision map and discuss phase portraits for no-slip billiards, including reviewing options for dealing with the dimensionality. 
We begin with a rapid overview of the derivation of the no-slip model. A detailed development can be found in \cite{CF}.

For any two bodies given by measurable sets in $\mathbb{R}^n$, assigning starting locations and orientations for reference, we may consider their images under all rigid translations and rotations, given by the special Euclidean group $SE(n)$. Excluding any configurations resulting in overlap (and certain restricted collisions) we have a configuration manifold $M \subset SE(n) \times SE(n)$ with boundary points corresponding to points of collision. Allowing the bodies to move freely between collisions, the system is then determined completely by the choice of a linear map $$\mathcal{C}: T_{\mathbf q} M^- \longrightarrow T_{\mathbf q} M^+$$ from the outward half tangent space to the inward for each $\mathbf q \in \partial M $. Under the natural requirements of conservation of total energy and momentum (allowing an exchange between linear and angular momentum), time reversibility, and the requirement that the collisions are rigid body collisions with the force applied at the point of contact, no-slip billiards naturally arise. For the two dimensional case, there are only two maps satisfying these requirements, namely the standard specular collision and the no-slip collision. 
However, for larger $n$ more options arise, loosely speaking based on the choices for \emph{rough} and \emph{smooth} directions. When we reference higher dimensional examples we will assume the map is the unique completely rough case.

While this collision model holds very generally, we now make several simplifying assumptions. Taking the large mass limit, one the bodies (henceforward the \emph{table}) may be viewed as fixed, while the finite mass body (the particle) will be assumed to be a sphere according to the dimension of the system, with rotationally symmetric mass distribution given by $\gamma$. We will assume a particle radius that is negligibly small but non-zero. In some (especially non-convex) tables this choice becomes more important, and indeed recent work has shown examples where the 
distinction between the ``mathematical'' point mass systems and the ``physical'' billiards with positive radius are significant \cite{AB, B}, creating regular or chaotic systems according to the case. However, for the examples we consider the distinction is not consequential.

Under these simplifying assumptions, then, if $\mathbf{p}^- \in T_{\mathbf q} M^-$ is expressed in the frame $(e_1,e_2,e_3)$ corresponding to the rotational, tangent, and outward normal directions, the two dimensional no-slip map  is given explicitly by $$\mathbf{p}^+=\mathcal{C}(\mathbf{p}^-)=T \mathbf{p}^-,$$ where
 $T$ is given by
\begin{equation*}\label{collision_map}\begingroup
\renewcommand*{\arraystretch}{1.5}
T= \left(\begin{array}{ccr}-\frac{1-\gamma^2}{1+\gamma^2} & -\frac{2\gamma}{1+\gamma^2} & 0 \\-\frac{2\gamma}{1+\gamma^2} & \ \, \, \frac{1-\gamma^2}{1+\gamma^2} & 0 \\0 & 0 & -1\end{array}\right),
\endgroup \end{equation*}
with $\gamma$ 
the previously defined mass distribution constant, and $T$ acts on the incoming velocity $\mathbf{p}^-$ through multiplication of the corresponding column vector.  It is convenient to express the transformation in terms of $\beta$ where $\gamma=\tan(\beta/2),$ in which case we have
\begin{equation}\label{eq:collision_map_beta}\begingroup
\renewcommand*{\arraystretch}{1.5}
T= \left(\begin{array}{ccr}-\cos\beta & -\sin\beta & 0 \\-\sin\beta &\ \, \cos\beta & 0 \\0 & 0 & -1\end{array}\right)
=\left(\begin{array}{cr} R(\beta) & 0 \\ 0 & -1\end{array}\right),
\endgroup \end{equation} 
where $R(\beta) \in O(2)$ is a $2\times 2$ rotation matrix, through which rotational and tangential components of the velocity may interact. 

\begin{figure}[htbp]
    \centering
    \includegraphics[width=.315\textwidth]{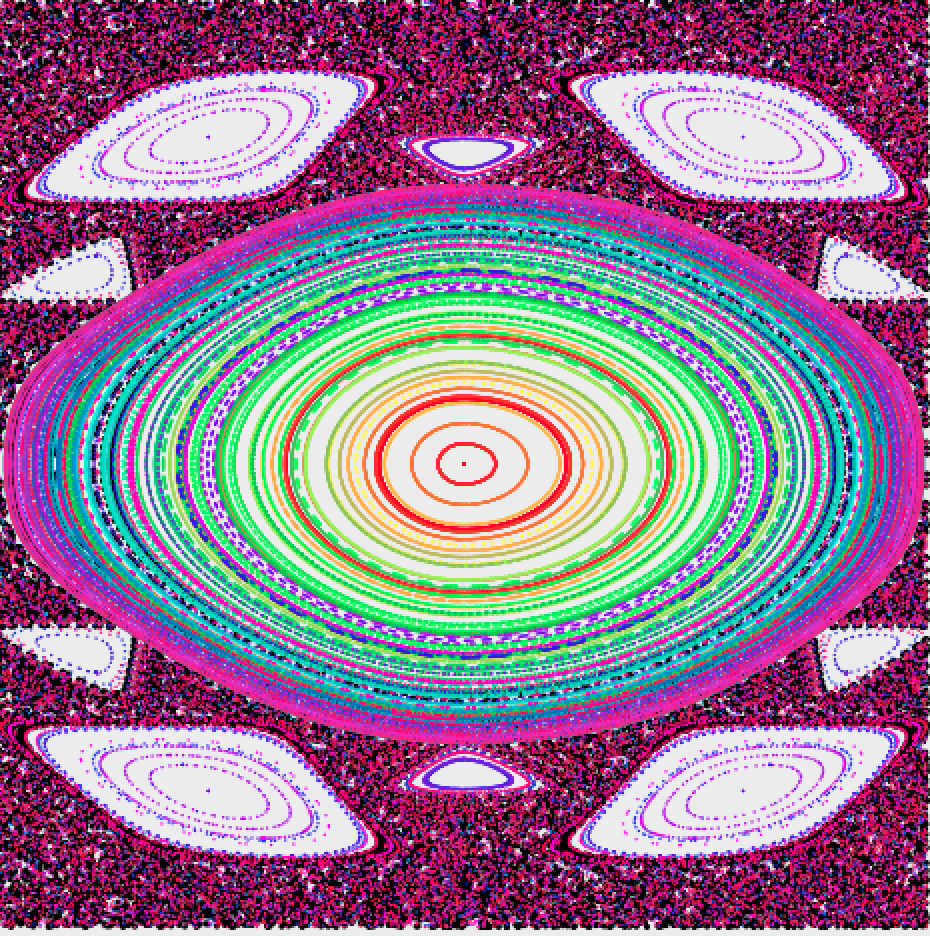}
    \includegraphics[width=.33\textwidth]{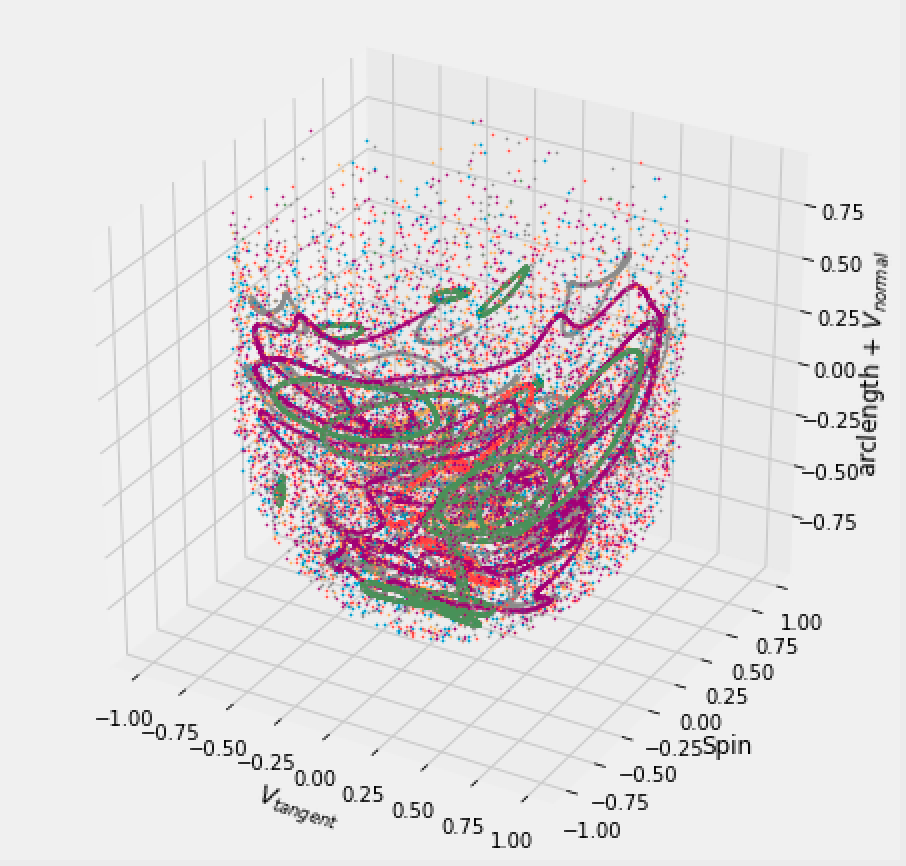}
    \includegraphics[width=.315\textwidth]{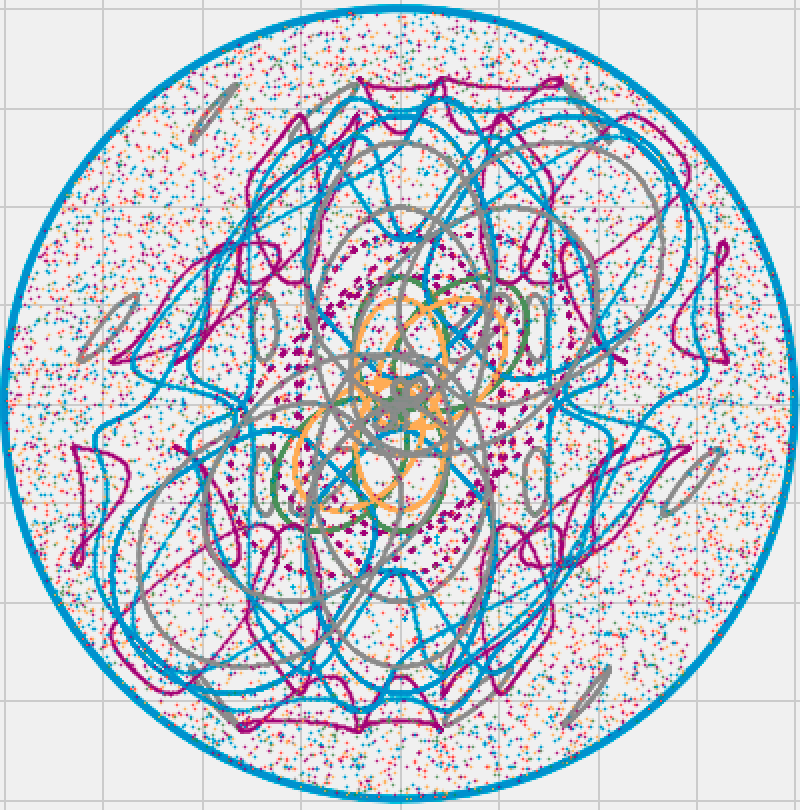}
   
\caption{\small The phase portrait of a specular billiard (left) is two dimensional, but the reduced no-slip phase portrait (center) incorporates angular velocity in addition to arclength and linear velocity. 
Often the velocity phase portrait projection (right)  is sufficient to show structural details.        }
    \label{fig:phaseportraits}
\end{figure}

Phase portraits are a valuable tool in qualitative analysis of specular billiards and may be extended to no-slip billiards. Here our focus is restricted to no-slip billiards in two spacial dimensions. The configuration space is then three dimensional due to the added rotational dimension. 
Recall that a priori, specular billiard systems in the plane are four dimensional with a state completely determined by a point 
in $\mathbb{R}^2 \times \mathbb{R}^2$.
From the perspective of the dynamics, there are two standard reductions that may be made without losing any essential information. 
First, one looks at the billiard map instead of the billiard flow by ignoring the paths between collisions and only record the position $s_n$ of the $n$th collision on the parametrized boundary $S^1$.  Second, one may require unit velocity, in which case the exiting velocity at a collision may be given by the angle $\phi_n$ relative to the inward normal. Note that the set of exit angles may be viewed as a half circle, that is, $S^1_+$.
The resulting phase portrait is two dimensional, as in Figure \ref{fig:phaseportraits} (left).

In the case of no-slip billiards, the unreduced phase space is six dimensional, which may be first reduced as in the specular case: 
 we keep track of the collisions $s_n \in S^1$ on the  parametrized boundary and we assume unit velocity. Note that the velocity vectors have a rotational velocity component, so the set of possible unit exit velocities corresponds to $S^2_+$. It turns out we can make one additional reduction, as the rotational position--though not the rotational velocity--is unimportant and may be ignored. Thus the \emph{reduced phase space} can be identified with $S^1 \times S^2_+$, a solid torus. A direct way to visualize the reduced phase space is to consider a solid cylinder with the ends identified, but a projection that sometimes simplifies the picture while retaining much of the essential information is the \emph{velocity phase portrait}, in which the position of collision is projected out and the velocities are all projected onto a disk.
Figure \ref{fig:phaseportraits} shows a  specular phase portrait for a lemon billiard, the three dimensional reduced phase portrait for a no-slip Sinai billiard comprised of a single scatterer on a torus, and the velocity phase portrait of the same no-slip billiard. One might also project onto the specular phase portrait, as in \cite{ACW}, flattening the cylinder, but we will consider only the velocity projection.

If one is familiar with  specular phase portraits, it is important to keep in mind that now closed invariant loops may overlap 
due to the projection. Nonetheless, a clear picture often emerges, showing for example an ergodic sea or in contrast a finely detailed structure, as in the polygons in  Figure  \ref{fig:vppregpoly}. In Section \ref{sec:inv} the velocity phase portraits of billiards with force will be useful in viewing the dissipation of the original structure and emergence of new invariant regions as the magnitude of the external force is increased.

\begin{figure}[htbp]
    \centering
    \includegraphics[width=\textwidth]{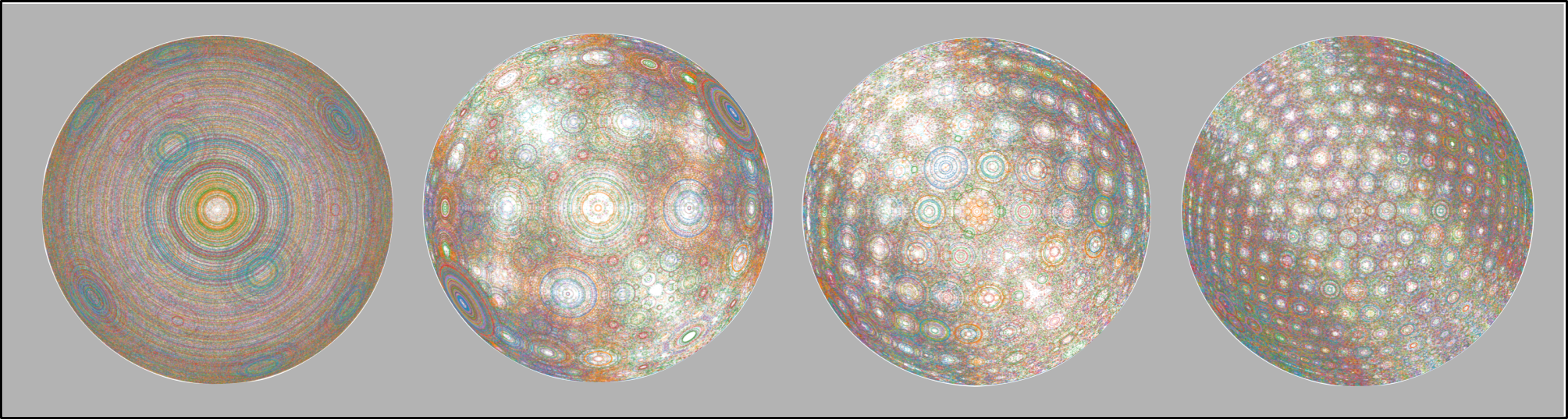}
   
    \caption{\small The velocity phase projections of regular  polygons with uniform mass distribution no-slip collisions, for a rectangular billiard (left), a hexagon (left center), a dodecagon (right center), and a icosagon (right). Distinct colors correspond to distinct orbits.
        }
    \label{fig:vppregpoly}
\end{figure}

\section{Stability of periodic points and existence of invariant regions for no-slip billiards with force}
\label{sec:inv}

One of the primary complications of incorporating no-slip collisions into the Lorentz gas model, even when no force is present, is that invariant regions are no longer precluded and the ergodicity that opens the door to statistical study in the specular case is not guaranteed. Using velocity phase portraits in Figure \ref{fig:lorentzgamma},  evidence of invariant regions appears as $\gamma$ increases. 
There are no visible signs of regularity for small $\gamma>0$, except for very large radius resulting in a nearly inscribed scatterer. However, it is known analytically that for any arbitrarily small $\gamma>0$ and any arbitrarily small radius of the scatterers there will exist correspondingly small invariant regions of phase space, located where the exit angle is very nearly tangent.  
\begin{figure}[htbp]
    
    \centering
    \includegraphics[width=\textwidth]{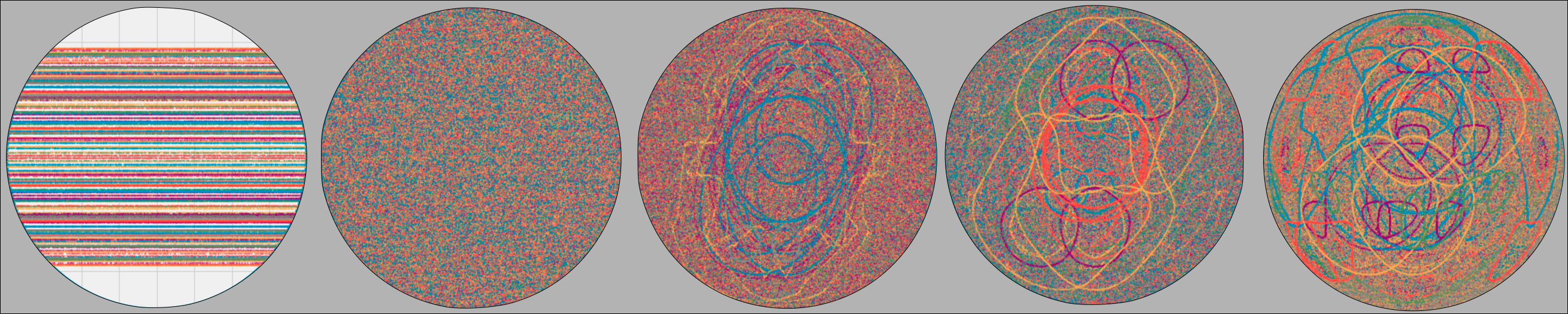}
   
    \caption{\small 
  Velocity phase portrait projections for a no-slip Sinai billiard 
  for $\gamma=0, 0.25, 0.5, 0.75,$ and $1$. The $\gamma=0$ case is dynamically equivalent to a specular billiard, and the regularity is a consequence of the constant rotational velocities. No signs of invariant regions (indicated by closed loops) appear until $\gamma>0.3$. 
     }
    \label{fig:lorentzgamma}
\end{figure}

Turning to questions of the stability of periodic points in no-slip billiards with an external force, we are ostensibly interested in Lyapunov stability, under which an orbit sufficiently close to a periodic point will remain close. However, even for no-slip billiards without force, an analytic demonstration of Lyapunov stability in the presence of circular scatterers is as yet unknown.  Accordingly, the results in this section  are in the form of conjectures based on numerical experiments, rendering more precise distinctions premature at this point.

The study of stability in scattering billiard systems with and without external forces and with specular or no-slip collisions can be divided into four regimes. Much more is known about the first three (systems with no force and either collision law, and specular systems with force), as we outline below, and give a preliminary study of the fourth regime (no-slip systems with force).

\begin{figure}[htbp]
    
    \centering
    \includegraphics[width=.85\textwidth]{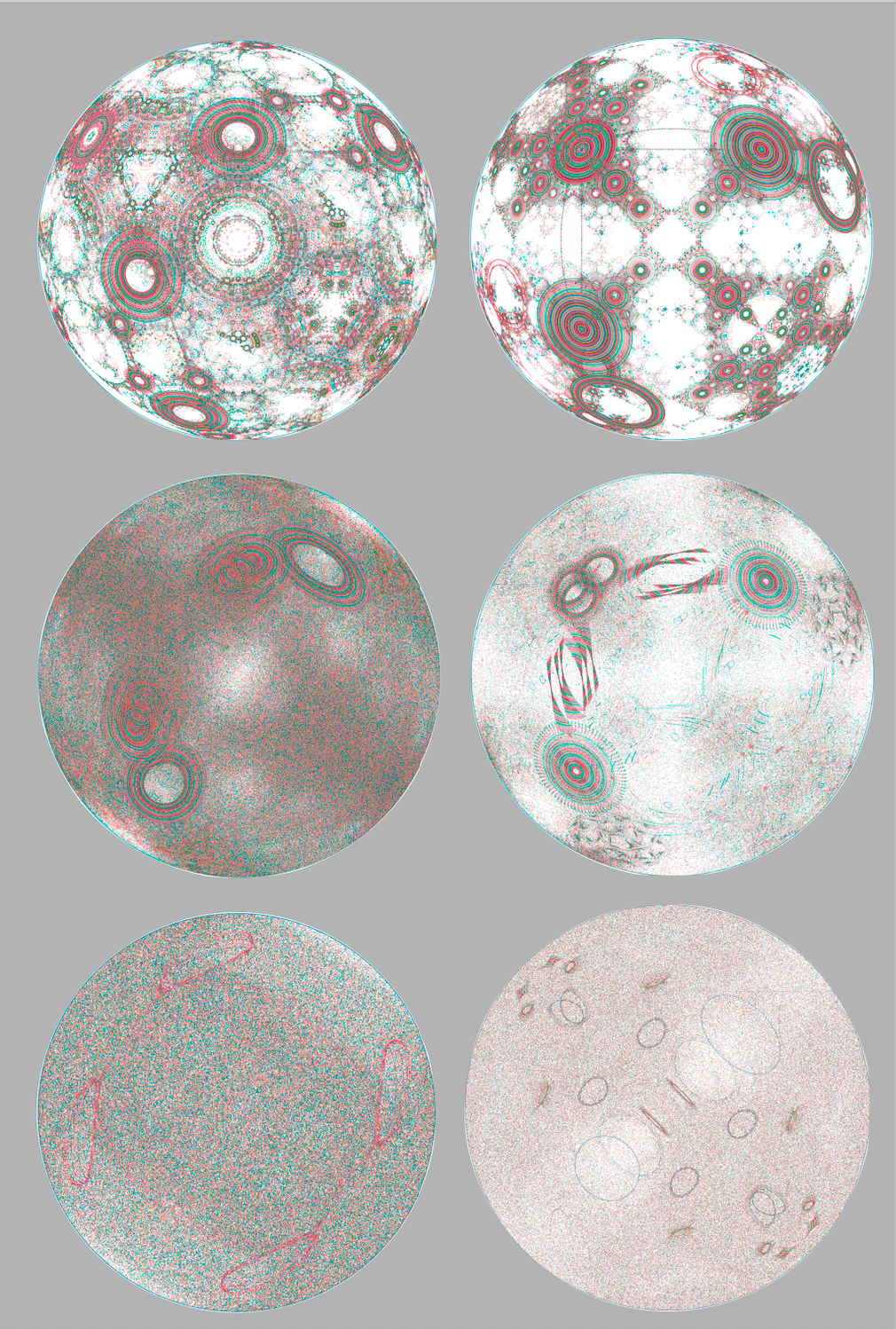}
    \caption{\small 
Velocity phase portraits for billiards with a disk-shaped scatterer, with  uniform mass $\gamma
=1/\sqrt{2}$
in the left column and  $\gamma=1$ in the right column. Even with no force (top row) an elaborate structure including elliptic islands is in evidence. This persists when a force is added (second row) though the structure dissipates as the force increases. With a stronger force (third row)  new invariant regions arise, analogous to those known to arise in the specular billiard case.
     }
    \label{fig:VPPforce}
\end{figure}

\begin{itemize}
\item \textit{Specular collisions with no force:} Much is known about the Sinai billiard, which can be tiled across the plane  to form the Lorentz gas model. In particular, its billiard map is known to be ergodic \cite{sinai}, ensuring no positive measure invariant regions.

\item \textit{No-slip collisions with no force:} There may be linearly stable periodic orbits between scatterers \cite{CFZ}, with a sharp threshold according to the curvatures of and distance between scattering disks \cite{W}. While it is not clear that linear stability implies Lyapunov stability, there is considerable numerical evidence of invariant regions forming elliptic islands, as seen in Figure \ref{fig:lorentzgamma}.

\item \textit{Specular collisions with force:} In \cite{Moran} 
    numerical  evidence is presented showing that invariant regions around near-periodic orbits may exist for specular collisions with force. In particular, see Figures 4 and 5 there. It was subsequently shown  
    that the attractor of the nonequilibrium Lorentz gas covers the whole of the accessible phase space for a small force, but at a certain threshold invariant pockets appeared \cite{dettmann}. 
    In \cite{Chernov2} such systems are studied analytically, and the evolution of trajectories  is described with a Sinai-Ruelle-Bowen measure, as  the requisite measure preservation for questions of ergodicity fails for the standard billiard measure once a force is incorporated.
\end{itemize}

It is therefore natural to fill in the fourth case of no-slip billiards with force by conjecturing that invariant regions will arise, as proposed in Conjecture \ref{conj:inv}, either due the persistence of elliptic periodic points in the no-slip regime or due to a process parallel to the one by which the attractor ceases to cover the whole of accessible phase space for the specular Lorentz gas. Looking at the velocity phase portraits in Figure \ref{fig:VPPforce}, it appears that indeed both possibilities materialize in the case of a scatterer in a hexagonal cell for the cases of uniform density ($\gamma 
=1/\sqrt{2}
$) and hollow ($\gamma=1$) colliding particles. Specifically, it appears that the ubiquitous structure present in the no-force case persists after a small force is introduced and then degrades as the force becomes larger, only to have a new structure appear, analogous  to the known behavior of  specular billiards.

\begin{figure}[htbp]
    
    \centering
    \includegraphics[width=.99\textwidth]{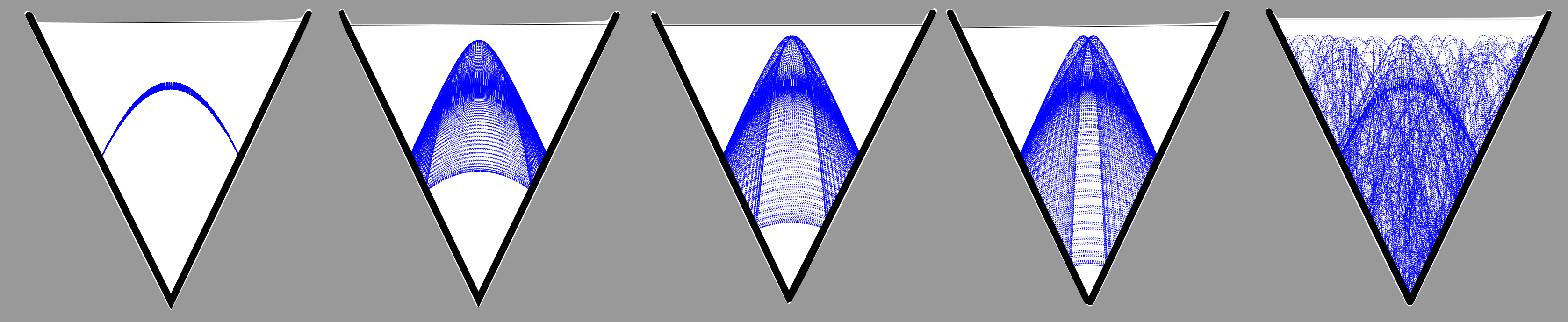}

    \caption{\small 
For a wedge with angle $2\pi/7$ orbits are perturbed slightly from their respective ideal $2$-periodic orbits, for launch angles $1.122, 1.123, 1.124, 1.125$, and $1.126$ radians. While the perturbation is identical the level of instability introduced for launch angle $1.126$ is more significant than for launch angle $1.122$, which remains stable.
     }
    \label{fig:2perunstable}
\end{figure}

We end this section with a closer look at the limits of stability and the factors leading to instability.
For no-slip wedges with no force, without the complication of curvature, all $2$-periodic points are known to be Lyapunov stable \cite{CFZ}. Experiments suggest that this property will sometimes, but by no means universally, carry over when a force is introduced for no-slip wedge billiards. Clearly, if the wedge opens downward a sufficient perturbation (or alternatively, a sufficient force) could result in a trajectory escaping the wedge entirely without further collisions, hence it is clear that the addition of a force may destroy stability in those cases.  Experiments suggest, however, that the creation of instability with the addition of force is more widespread, including for trajectories where the wedge opens upward.  If we fix a wedge angle and collision points, then consider an orbit slightly perturbed from the $2$-periodic requirement in Theorem \ref{thm:periodic}, varying the launch angle $\theta$ (while maintaining a constant relative perturbation) at  certain angles the invariance rapidly dissipates, as seen in Figure \ref{fig:2perunstable}. Accordingly, we make the following conjecture. 

\begin{conjecture}
The stability of path-reversing $2$-periodic orbits in 
no-slip billiards with force
depends on the launch angle.
\label{conj:launchstability}
\end{conjecture}

As our experimental approach to the launch angle question is equally amenable to billiards more general than wedges, we will allow curvature and consider a billiard with two identical scatterers. 
Even when curvature is included, as in a no-slip Lorentz gas, the limits of linear stability are well understood in the absence of a force. 
Suppose we have two identical radius $R$ disks with curvature $\kappa=1/R$ and and opposite collision points distance $d$ apart, such that their tangential wedge has angle $2\phi$. 
Then the $2$-periodic point will be elliptic whenever
\begin{equation}
\kappa d < \frac{2-2\cos^2(\beta/2)\cos^2\phi}{\cos^2(\beta/2)\cos\phi}
    \label{eq:ellipticity}
\end{equation}
where $\beta$ is the alternate mass distribution constant \cite{CFZ}.

\begin{figure}[ht]
    
    \centering
    \includegraphics[width=.99\textwidth]{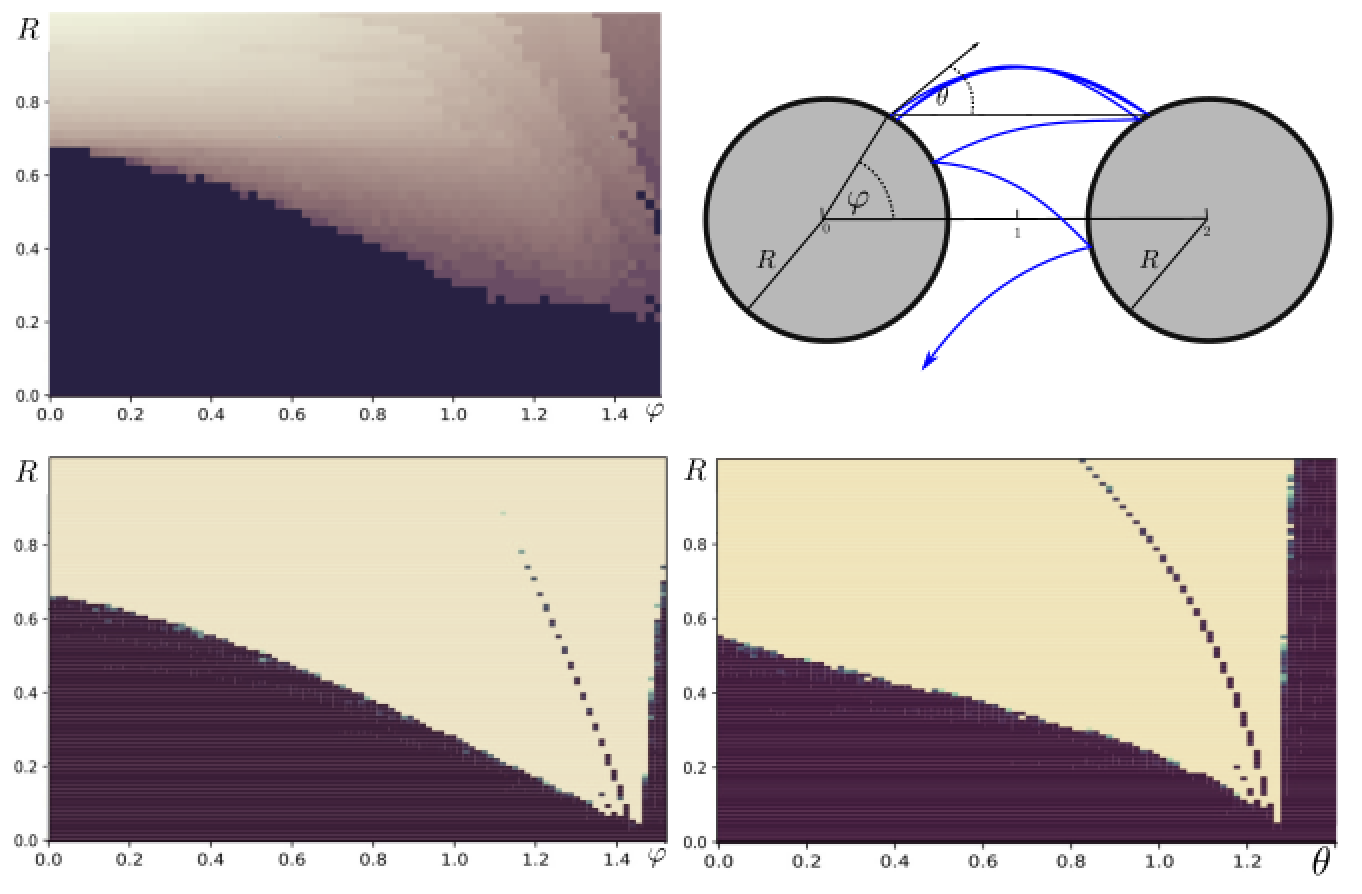}

    \caption{\small Dark regions indicating experimental stability of periodic points between identical scatterers. Upper left: With no force, stability of horizontal $2$-periodic  orbits depends only on curvature $\kappa=1/R$ and distance, parametrized by $\varphi.$ Upper right: With force, we redo the experiment for trajectories with positive launch angle $\theta$ slightly perturbed from the ideal $2$-periodic orbit. Lower left: The stability curve for the force case is similar to the no-force case. Lower right: If we fix $\varphi$, stability is determined by $\kappa=1/R$ and launch angle $\theta$.
     }
    \label{fig:stabexp}
\end{figure}

Figure \ref{fig:stabexp} gives the set up and results for the following numerical experiments.
Assume the distance between the centers of the disks is 2, and they have the same radius $0< R < 1$. We will vary the curvature by allowing $R$ to vary, and we will vary the distance by moving $\mathbf{q}_0$ around the first quadrant of the scatterer, parametrized by the angle $\varphi$ relative to the horizontal, with $\mathbf{q}_1$ situated at the corresponding point on the second scatterer. Sampling 100 values of $R$ and 100 values of $\varphi$, we consider an orbit slightly perturbed from the $2$-periodic orbit for each pair and record the number of collisions before it fails to hit one of the scatterers, or reaches a maximum of 1000 collisions. We then graph the results, with darker points indicating higher stability as shown by a large number of collisions.

As a base case,  the first experiment (Figure \ref{fig:stabexp}, upper left) considers horizontal $2$-periodic orbits in the case with no force. Equation \eqref{eq:ellipticity} predicts a threshold according to a cosine curve, and the experimental results are as expected, with the caveat that the samples near $\varphi = \pi/2$, where the orbits are nearly tangent, have a noticeable experimental error. The second experiment (Figure \ref{fig:stabexp}, lower left) is identical to the first, except that a force is added and the orbits are adjusted according to the formula in Theorem \ref{thm:periodic}. It appears that Equation \eqref{eq:ellipticity} continues to hold and the stability threshold is unaltered.

For the final experiment (Figure \ref{fig:stabexp}, lower right) we fix $\varphi=0.4$ and instead vary only the launch angle $\theta$.
The results suggest that in the no-slip case the launch angle may in fact determine stability, and in fact it appears that for a fixed curvature (or radius) there may be transition from stable to unstable and back to stable as the launch angle is increased from $0$ to $\pi/2$. In this case, based on checking individual orbits for launch angles around $1.3$, it appears that the emerging stable region for nearly vertical launch angles is a real effect and not merely an experimental artifact.

\section{Proofs of Theorems} 
\label{sec:bounce}
 
We now prove Theorem \ref{thm:bounce} characterizing the no-slip bouncing ball. 
\begin{proof}[Proof of Theorem \ref{thm:bounce}]
First, we demonstrate that the velocity is $2$-periodic.
Again letting $e_1,e_2,e_3$ correspond to the rotational, tangent, and outward normal directions, let  $R_\theta$ be the matrix which rotates $e_2, e_3$ by $\pi$ and fixes $e_1$, then by Equation \eqref{eq:collision_map_beta} the change in velocity between bounces is
\begin{equation*}\label{eq:collision_map_beta_bounce}\begingroup
\renewcommand*{\arraystretch}{1.5}
S=R_{\pi}TR_{\pi}= \left(\begin{array}{rcr}-\cos\beta & \sin\beta & 0 \\\sin\beta &\ \, \cos\beta & 0 \\0 & 0 & -1\end{array}\right).
\endgroup \end{equation*}  
A short calculation shows $S^2 = I$,
and the velocity orbit is $2$-periodic for any initial velocity. 

Next we consider the question of determining under what conditions the particle returns to the same position.
One might directly solve for the displacement after two collisions; alternatively, we note that by symmetry the second bounce will reverse the first if and only if 
$$ p_+ = Sp_-=-p_-.$$
This yields the system
\begin{align*}
    -\dot{x}_0\cos\beta  + \dot{y}_0\sin\beta &= -\dot{x}_0\\
    \dot{x}_0\sin\beta + \dot{y}_0\cos\beta  &= -\dot{y}_0,
\end{align*}
with the solution 
\begin{equation*}
\label{eq:2perbouncebeta}
\Delta=\frac{\dot{x}_0}{\dot{y}_0}=\frac{1+\cos \beta}{\sin \beta}.
\end{equation*}

\begin{figure}[htbp]
    
    \centering
    \includegraphics[width=.99\textwidth]{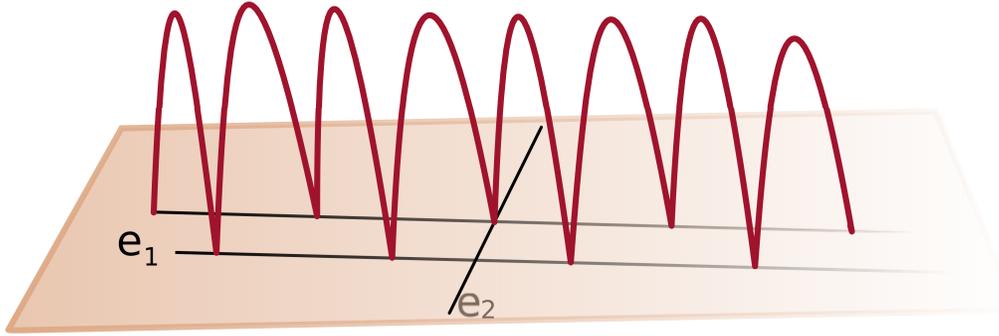}

    \caption{\small 
A simulation of three dimensional no-slip bouncing of a small radius sphere on a plane. For any initial conditions not resulting in $2$-periodic motion, there will be a tangential direction $e_1$ along which there is drift and an orthogonal direction $e_2$ in which the reduced system is $2$-periodic.
     }
    \label{fig:3dbounce}
\end{figure}

Recalling that  $\beta=2\arctan\gamma$, we may rewrite the condition for periodicity in terms of the mass distribution constant as
simply $\Delta = \gamma$,
the condition required in Theorem \ref{thm:bounce}, and the result follows.
\end{proof}

Figure \ref{fig:3dbounce} shows a general example of a three dimensional sphere bouncing in the plane with no-slip collisions, and indeed the theorem generalizes to a billiard with a single hyperplane boundary and orthogonal force in any dimension. 
Extending the transformation matrix 
in Equation \eqref{eq:collision_map_beta_bounce}
to a spherical particle in a three dimensional system, we may choose coordinates so that the change in velocity is give by 
\begin{equation*}\label{collision_3dblock}\begingroup
\renewcommand*{\arraystretch}{1.5}
T= \left(\begin{array}{ccc} R(\beta_1) & 0 & 0 \\ 0 &\ \, R(\beta_2) & 0 \\0 & 0 & J\end{array}\right),
\endgroup \end{equation*} 
where $R(\beta_i) \in O(2)$  and
\[\renewcommand*{\arraystretch}{1.5}
J=\left(\begin{array}{rc}
-1 & 0 \\
0 & 1
\end{array}\right).
\]
Notice that for a general table these decompositions are of limited value for understanding long-term behavior because the form will be destroyed by the necessary change of frame between collisions. However, as discussed in \cite{CCCF}, for a cylinder in dimension three or greater separation of the axial direction persists and therefore the larger system separates into motion along the cylindrical axis, mirroring a two dimensional no-slip billiard, and the remaining motion which is restricted to the cross-section. Here we are interested in the special case were the only boundary is a hyperplane,  hence each cross-sectional system is in turn a cylinder, and by repeated decomposition each tangential direction mirrors a two dimensional system.
The periodicity of the velocity orbit observed in the proof of Theorem \ref{thm:bounce} then extends to bouncing in any dimension on the corresponding hyperplane, and the subsequent argument may be applied directly to $\Delta_1, \Delta_2, \dots, \Delta_k$, the tangential-rotational ratios in all the tangential directions.
Then the motion will be periodic if $\Delta_i=\gamma$ for all $i$, while there will be drift in any tangential direction where there is not equality, and result follows in any dimension. With the appropriate decomposition, in fact, in the nonperiodic case one might choose coordinates to give a single tangential direction in which the drift occurs, while the system is $2$-periodic in all the remaining directions, generalizing the process illustrated in Figure \ref{fig:3dbounce}.

We now prove Theorem \ref{thm:periodic},
giving conditions for a $2$-periodic orbit between opposite boundary points $\mathbf{q}_0$ and $\mathbf{q}_1$ for a trajectory with launch angle $\theta$ and tangential wedge angle $\phi$.
\begin{proof}[Proof of Theorem \ref{thm:periodic}]
We begin by
observing that we must have speed
$$v=2\sqrt{\frac{gd}{\sin{2\theta}}}$$ as in Equation \eqref{eq:vel}  to ensure that a particle leaving $\mathbf{q}_0$ at an angle $\theta$ 
relative to the horizontal will next collide at the corresponding point $\mathbf{q}_1$ on the tangential wedge, with incoming velocity $\mathbf{v}_1^-=(\dot{x}_0, v\cos\theta, v\sin\theta)$, for some initial rotational velocity $\dot{x}_0$.
By symmetry, if $\dot{x}_0$ may be chosen so that
$$R_{\phi}TR_{-\phi} \mathbf{v}_{1}^-=\mathbf{v}_1^+=(-\dot{x}_0, -v\cos\theta, v\sin\theta),$$ then the orbit will be $2$-periodic. By Equation \eqref{eq:collision_map_beta}, using the mass distribution constant $\beta$, this is equivalent to the system
\begin{align*}
\dot{x}_0 & = \dot{x}_0\cos\beta +v\cos\phi\sin\theta\sin\beta-v\sin\phi\cos\theta\sin\beta \nonumber \\
 v\cos\phi & =\dot{x}_0\sin\theta\sin\beta-v\cos\phi\sin^2\theta\cos\beta+v\cos^2\phi-v\sin\phi\sin\theta\cos\theta(1+\cos\beta) \nonumber\\
 v\sin\theta & =-\dot{x}_0\cos\theta\sin\beta+v\cos\phi\sin\theta\cos\theta(1+\cos\beta)+v\sin\phi(\cos^2\theta + \sin\theta)\cos\beta.
\end{align*}

Solving the first equation for $\dot{x}_0$ gives the requirement
\begin{equation*}
    \dot{x}_0=\frac{v\cos\phi\sin\theta\sin\beta-v\sin\phi\cos\theta\sin\beta}{1-\cos\beta}
\end{equation*}
whence
\begin{equation*}
    \frac{v}{\dot{x}_0}(\cos\phi\sin\theta-\sin\phi\cos\theta)=\frac{1-\cos\beta}{\sin\beta}=\tan(\beta/2)=\gamma,
\end{equation*}
and Equation \eqref{eq:2per} is the necessary condition. It can be verified that this choice of $\mathbf{q}_0$ also satisfies the two remaining equations describing linear velocity. Finally, the same argument holds if the wedge is opening downward, except that to get a real solution one must change the sign of the force, and Corollary \ref{cor:opendown} follows.
\end{proof}

\section{Source code} 
Numerical experiments were conducted using code written in Python and SageMath, which is available upon request. 

\section{Acknowledgements}

The authors would like to thank the Tarleton University Office of Research and Innovation, who provided support through a First Year Student Experience grant and Faculty Student Research Grants, the American Mathematical Society and National Science Foundation, who provided support through an NREUP grant, the University of Delaware, who provided support through the Summer Scholar program, and  Mount Holyoke College, who provided travel support.


\begin{thebibliography}{10}

\bibitem{ACW}
J.~Ahmed, C.~Cox, and B.~Wang.
\newblock No-slip billiards with particles of variable mass distribution.
\newblock {\em Chaos: An Interdisciplinary Journal of Nonlinear Science},
  32(2):023102, 2022.

\bibitem{AB}
H.~Attarchi and L.~A. Bunimovich.
\newblock Collision of a hard ball with singular points of the boundary.
\newblock {\em Chaos: An Interdisciplinary Journal of Nonlinear Science},
  31(1):013123, 2021.

\bibitem{BKM}
A.~V. Borisov, A.~A. Kilin, and I.~S. Mamaev.
\newblock On the model of non-holonomic billiard.
\newblock {\em Regul. Chaotic Dyn.}, 16(6):653--662, 2011.

\bibitem{BGKT}
M.~Boshernitzan, G.~Galperin, T.~Kr{\"u}ger, and S.~Troubetzkoy.
\newblock Periodic billiard orbits are dense in rational polygons.
\newblock {\em Transactions of the American Mathematical Society},
  350:3523--3535, 1998.

\bibitem{gutkin}
D.~S. Broomhead and E.~Gutkin.
\newblock The dynamics of billiards with no-slip collisions.
\newblock {\em Phys. D}, 67(1-3):188--197, 1993.

\bibitem{B}
L.~A. Bunimovich.
\newblock Physical versus mathematical billiards: From regular dynamics to
  chaos and back.
\newblock {\em Chaos: An Interdisciplinary Journal of Nonlinear Science},
  29(9):091105, 2019.

\bibitem{Chernov2}
N.~Chernov.
\newblock Sinai billiards under small external forces.
\newblock {\em Annales Henri Poincare}, 2:197--236, 06 2001.

\bibitem{CD}
N.~Chernov and D.~Dolgopyat.
\newblock The {G}alton board: limit theorems and recurrence.
\newblock {\em J. Amer. Math. Soc.}, 22(3):821--858, 2009.

\bibitem{CCCF}
T.~{Chumley}, S.~{Cook}, C.~{Cox}, and R.~{Feres}.
\newblock {Rolling and no-slip bouncing in cylinders}.
\newblock {\em J. Geom. Mech.}, 12(1):53--84, 2020.

\bibitem{CZ15}
M.~F. Correia and H.-K. Zhang.
\newblock Stability and ergodicity of moon billiards.
\newblock {\em Chaos: An Interdisciplinary Journal of Nonlinear Science},
  25(8):083110, 2015.

\bibitem{CF}
C.~Cox and R.~Feres.
\newblock Differential geometry of rigid bodies collisions and non-standard
  billiards.
\newblock {\em Discrete Contin. Dyn. Syst.}, 36(11):6065--6099, 2016.

\bibitem{CFZ}
C.~Cox, R.~Feres, and H.-K. Zhang.
\newblock Stability of periodic orbits in no-slip billiards.
\newblock {\em Nonlinearity}, 31(10):4443--4471, 2018.

\bibitem{CFBZ}
C.~Cox, R.~Feres, and B.~Zhao.
\newblock Rolling systems and their billiard limits.
\newblock {\em Regular and Chaotic Dynamics}, 26(1):1--21, 2021.

\bibitem{cross}
R.~Cross.
\newblock Grip-slip behavior of a bouncing ball.
\newblock {\em American Journal of Physics}, 70(11):1093--1102, 2002.

\bibitem{dettmann}
C.~P. Dettmann.
\newblock Diffusion in the {L}orentz gas.
\newblock {\em Commun. Theor. Phys. (Beijing)}, 62(4):521--540, 2014.

\bibitem{DM}
C.~P. Dettmann and G.~P. Morriss.
\newblock Crisis in the periodic lorentz gas.
\newblock {\em Phys. Rev. E}, 54:4782--4790, Nov 1996.

\bibitem{galton}
F.~Galton.
\newblock Natural inheritance, 1889.

\bibitem{Garwin}
R.~L. Garwin.
\newblock Kinematics of an ultraelastic rough ball.
\newblock {\em American Journal of Physics}, 37(1):88--92, 1969.

\bibitem{gualtieri}
M.~Gualtieri, T.~Tokieda, L.~Advis-Gaete, B.~Carry, E.~Reffet, and C.~Guthmann.
\newblock Golfer’s dilemma.
\newblock {\em American Journal of Physics}, 74(6):497--501, 2006.

\bibitem{hefner}
B.~T. Hefner.
\newblock The kinematics of a superball bouncing between two vertical surfaces.
\newblock {\em American Journal of Physics}, 72(7):875--883, 2004.

\bibitem{KP}
S.~O. Kamphorst and S.~Pinto-de Carvalho.
\newblock The first {B}irkhoff coefficient and the stability of 2-periodic
  orbits on billiards.
\newblock {\em Experiment. Math.}, 14(3):299--306, 2005.

\bibitem{kozlov}
V.~V. Kozlov and M.~Y. Mitrofanova.
\newblock Galton board.
\newblock {\em Regul. Chaotic Dyn.}, 8(4):431--439, 2003.

\bibitem{KR}
P.~L. Krapivsky and S.~Redner.
\newblock Slowly divergent drift in the field-driven lorentz gas.
\newblock {\em Phys. Rev. E}, 56:3822--3830, Oct 1997.

\bibitem{Lorentz}
H.~Lorentz.
\newblock The motion of electrons in metallic bodies i.
\newblock In {\em KNAW, proceedings}, volume~7, pages 438--453, 1905.

\bibitem{Moran}
B.~Moran, W.~G. Hoover, and S.~Bestiale.
\newblock Diffusion in a periodic {L}orentz gas.
\newblock {\em J. Statist. Phys.}, 48(3-4):709--726, 1987.

\bibitem{sinai}
J.~G. Sina\u{\i}.
\newblock Dynamical systems with elastic reflections. {E}rgodic properties of
  dispersing billiards.
\newblock {\em Uspehi Mat. Nauk}, 25(2 (152)):141--192, 1970.

\bibitem{W}
M.~P. Wojtkowski.
\newblock The system of two spinning disks in the torus.
\newblock {\em Phys. D}, 71(4):430--439, 1994.

\end{thebibliography}
\end{document}